\documentclass[a4,12pt]{article}%
\usepackage[margin=1.2in]{geometry}

\usepackage[dvipsnames]{xcolor}

\usepackage{graphicx}
\usepackage{graphics}

\usepackage{amsmath}
\usepackage{amsfonts}
\usepackage{amssymb}
\usepackage{enumerate}
\newtheorem{theorem}{Theorem}[section]

\newtheorem{definition}[theorem]{Definition}
\newtheorem{example}[theorem]{Example}

\newtheorem{lemma}[theorem]{Lemma}

\newtheorem{proposition}[theorem]{Proposition}

\newtheorem{remark}[theorem]{Remark}

\newenvironment{proof}[1][Proof]{\noindent\textbf{#1.} }{\ \rule{0.5em}{0.5em}}

\newcommand{\prob}{{\mathbb P}}

\newcommand{\dN}{{\mathbb N}}

\newcommand{\dR}{{\mathbb R}}

\newcommand{\dP}{{\rm \bf P}}
\newcommand{\dE}{{\rm \bf E}}

\newcommand{\ep}{\varepsilon}
\newcommand{\la}{\lambda}
\newcommand{\Eqref}[1]{Eq.~(\ref{#1})}

\newcommand{\black}{\color{black}}

\newcounter{figurecounter}
\begin{document}

\title{Absorbing Blackwell Games%
\thanks{
This work has been partly supported by COST Action CA16228
European Network for Game Theory. 
Ashkenazi-Golan 
acknowledges the support of the Israel Science Foundation, grants
\#217/17 and \#722/18, and the NSFC-ISF Grant \#2510/17. 
Solan acknowledges the support of the Israel Science Foundation, grant
\#217/17.}}
\author{Galit Ashkenazi-Golan%
\footnote{London School of Economics and Political Science,
Houghton Street,
London,
WC2A 2AE, UK,
E-mail: galit.ashkenazi@gmail.com.}\and
J\'{a}nos Flesch\footnote{Department of Quantitative Economics,
Maastricht University, P.O.Box 616, 6200 MD, The Netherlands. E-mail: j.flesch@maastrichtuniversity.nl.}
\and
Eilon Solan\footnote{School of Mathematical Sciences, Tel-Aviv University, Tel-Aviv, Israel, 6997800, E-mail: eilons@tauex.tau.ac.il.}}
\maketitle

\begin{abstract}
It was shown in Flesch and Solan (2022) with a rather involved proof that all two-player stochastic games with finite state and action spaces and shift-invariant payoffs admit an $\ep$-equilibrium, for every $\ep>0$. 
Their proof also holds for two-player absorbing games with tail-measurable payoffs.
In this paper we provide a simpler proof for the existence of $\ep$-equilibrium in two-player absorbing games with tail-measurable payoffs,
by combining recent mathematical tools for such payoff functions with classical tools for absorbing games.
\end{abstract}

\section{Introduction}
Flesch and Solan (2022) studied two-player stochastic games with finite state and action spaces, and showed that these games admit an $\ep$-equilibrium, for every $\ep>0$, provided that the payoff functions are bounded, Borel-measurable, and shift-invariant. 
The proof is involved and heavily relies on the results of Vieille (2000a, 2000b). The last condition of shift-invariance requires that the payoffs do not depend on the initial segment of the infinite play; the payoff remains the same if we replace an initial segment of the infinite play by another segment of any finite length. This condition is satisfied by several payoff functions, 
such as the long-term average of stage-payoffs (see, e.g., Sorin (1992) and Vieille (2000a, 2000b)), 
the limsup or liminf of stage-payoffs (see, e.g., Maitra and Sudderth (1993)), and 
various winning conditions studied in computer science, including the B\"{u}chi, parity, and M\"{u}ller winning conditions (see, e.g., Gr\"{a}del and Ummels (2008) and Chatterjee and Henzinger (2012)).

A focal class of stochastic games are absorbing games, which played a fundamental role in the development of stochastic games (see, e.g., Blackwell and Ferguson (1968), Vrieze and Thuijsman (1989), Solan (1999),
Solan and Vieille (2001), Solan and Vohra (2001, 2002), and Solan and Solan (2020, 2021)). These are stochastic games in which there is only one nonabsorbing state, and in the absorbing states the players have only one action. That is, the state of the game can change at most once, and once absorption occurs, the play is strategically over. 
In the papers mentioned above, 
various techniques have been developed to study 
different equilibrium concepts in
absorbing games,
and some of those have been extended to stochastic games with more than one nonabsorbing state
or to stopping games, see Shmaya and Solan (2004), Heller (2012), and Solan (2018).

In this paper we study absorbing games with bounded and Borel-measurable payoffs. To distinguish these from classical absorbing games, in which the payoff is the discounted sum of stage-payoffs or the long-term average of stage-payoffs, we call these games \emph{absorbing Blackwell games}, referring to Blackwell (1969), who proposed to look at simultaneous-move games with general evaluations.

The result of Flesch and Solan (2022) implies 
in particular
that two-player absorbing Blackwell games with finitely many actions and bounded, Borel-measurable, and shift-invariant payoffs admit an $\ep$-equilibrium, for every $\ep>0$. 
Following closely their proof reveals that in fact it extends to two-player absorbing Blackwell games with finitely many actions and bounded, Borel-measurable, and tail-measurable payoffs.
Tail-measurability requires that
the payoff remains the same if we replace an initial segment of the infinite play by another segment of the same length.
One example of a tail-measurable payoff function that is not shift-invariant is the function that is equal to the limsup of the average payoffs in stages that are even numbers.

Since the proof of Flesch and Solan (2022) applies to stochastic games with any finite number of nonabsorbing states,
its proof is involved.
In this paper we concentrate on two-player absorbing games with finite action sets and bounded, Borel-measurable, and tail-measurable payoffs, and provide a simpler proof for the existence of an $\ep$-equilibrium, for every $\ep > 0$.

The new proof has several advantages.
Since our proof is much simpler than the one provided in Flesch and Solan (2022),
it is more accessible to the community.
In addition, our proof does not rely on Vieille (2000a, 2000b), which apply to all two-player stochastic games, but rather on Vrieze and Thuijsman (1989), which is tailored specifically to two-player absorbing games,
and hence it connects the proof to classical tools in the study of equilibria in absorbing games.
As a consequence,
our proof method may pave the way for extentions of our result
 to more than two players or to more than one nonabsorbing state, by applying techniques developed for absorbing games with more than two players 
(e.g., Solan (1999), Solan and Vohra (2002), or
Solan and Solan (2021)) or more than one nonabsorbing state
(see Solan (2000)).

\medskip
\noindent\textbf{Related literature.} 
Blackwell (1969) presented a new class of two-player repeated games. In these games, at each stage the players choose simultaneously actions in their action sets, player~1 wins if the generated play is in some given set of plays, and player~2 wins otherwise.
Blackwell (1969) proved that if the action sets are finite and the winning set is $G_\delta$, that is,
a countable intersection of open sets, then the game has a value.
The restriction to finite action sets is natural, since even in one-shot games
with countable action sets, the value need not exist.
Orkin (1972) extended the result to include winning sets
in the Boolean algebra generated by $G_\delta$ sets,
and Vervoort (1996) extended it to winning sets that are $G_{\delta\sigma}$, that is, countable unions of $G_\delta$ sets.
Martin (1998) proved that the value exists as soon as the action sets are finite and the winning set is Borel-measurable.
In fact, Martin's (1998) result does not assume that the outcome is determined by a winning set;
rather, as soon as the outcome is a bounded Borel-measurable function of the whole play, the value exists.

A natural question that arises is whether every multiplayer nonzero-sum Blackwell game admits an $\ep$-equilibrium,
for every $\ep > 0$.
If the payoff functions of the players are tail-measurable,
then the answer is positive, see Ashkenazi-Golan, Flesch, Predtetchinski, Solan~(2022).
Blackwell games can be seen as stochastic games with a countable state space and deterministic transitions, by identifying each history of the Blackwell game with a state.
The existence of an $\ep$-equilibrium in this class of games is not known in general.
With regards to the discounted payoff,
Fink (1964) and Takahashi (1964) proved the existence of a discounted equilibrium when the set of states is countable and the action sets are compact,
under proper continuity conditions on the payoff functions and the transition function.
Vieille (2000a, 2000b) proved that every two-player stochastic game with finite state and action sets
admits an $\ep$-equilibrium for the average payoff, for every $\ep > 0$.
Solan (1999) proved an analogous result in three-player absorbing games with finite action sets.
Additional results in this direction can be found in, e.g.,
Solan and Vieille (2001), Simon (2007, 2012, 2016), Flesch et al.~(2008, 2009), Solan and Solan (2020), and Solan et al.~(2020);
see also the survey by Ja\'{s}kiewicz and Nowak (2016). For existence results for $\ep$-equilibrium with payoff functions other than the discounted or average payoff, see for example Maitra and Sudderth (1998, 2003), Chatterjee (2005, 2006), Le Roux and Pauly (2014), Flesch, Herings, Maes, Predtetchinski (2022).

To our knowledge, this is the first paper that combines Blackwell games and absorbing games, both in model and tools.

\medskip

\noindent\textbf{Structure of the paper.} The paper is organized as follows. In Section \ref{Sec.Model} we describe the model of absorbing Blackwell games. In Section \ref{sec-mainres} we present the main theorem. In Section \ref{sec-aux} we discuss preliminaries for the proof. In Section \ref{sec-proof} we prove the main theorem. 
In Section \ref{Sect.Discuss} we 
raise some open problems.

\section{Model}\label{Sec.Model}

In this section we present the model of absorbing Blackwell games. Then, we describe two classical classes of absorbing Blackwell games. Finally, we define tail-measurability, which is a crucial condition for our main theorem. 

\subsection{Absorbing Blackwell games}\label{sect-AbsBl}

Throughout the paper, the set of natural numbers is denoted by $\dN=\{1,2,\ldots\}$.

\begin{definition}\label{def-abs-Blackwell}
An \emph{absorbing Blackwell game} is a tuple \[\Gamma = (I,(A_i)_{i \in I}, (f_i, g_i)_{i \in I}, p),\]
where
\begin{enumerate}
    \item $I$ is a nonempty and finite set of players.
    \item For each $i\in I$, $A_i$ is a nonempty and finite action set for player $i$.\vspace{0.1cm}\\ Let $A=\times_{i\in I} A_i$ denote the set of action profiles; 
    let $A^{<\dN}$ denote the set of histories, i.e., finite sequences of elements of $A$, including the empty sequence; and let $A^\dN$ denote the set of runs, i.e., infinite sequences of elements of $A$. We endow $A$ with the discrete topology and $A^\dN$ with the product topology.
    \item For each $i \in I$, $f_i : A^{<\dN} \to \dR$ is a bounded function, called the  absorbing payoff function for player~$i$. 
    \item For each $i \in I$, $g_i : A^{\dN} \to \dR$ is a bounded and Borel-measurable function, called the nonabsorbing payoff function for player $i$.
  \item $p : A \to [0,1]$ is a function assigning the probability of absorption $p(a)$ to each action profile $a\in A$.
\end{enumerate}
\end{definition}

We assume w.l.o.g.~that $f_i$ and $g_i$, for each player $i\in I$, take values in $[0,1]$. 

The game is played 
in discrete time
as follows. At every stage $t\in\dN$, each player $i\in I$ chooses an action $a_{ti} \in A_i$,
simultaneously with the other players. This yields an action profile $a_t=(a_{ti})_{i\in I}$. With probability $p(a_t)$ the play terminates
(in equivalent terminology, the play is absorbed) and the terminal payoff is $f_i(a_1,\ldots,a_t)$ to each player $i\in I$.
With probability $1-p(a_t)$ the play continues to the next stage $t+1$.
If the play is never absorbed, the payoff is $g_i(a_1,a_2,\ldots)$ to each player $i\in I$.

\medskip

\noindent\textbf{Histories.} A \emph{history} at stage $t$ is $(a_1,a_2,\ldots,a_{t-1})$, and it represents a game-situation in which absorption has not taken place.\footnote{Thus, we assume that the play of the game continues to stage $t$ under $(a_1,a_2,\ldots,a_{t-1})$, even if absorption has probability~1 under this history. This is merely done to simplify the exposition.} A generic history, i.e., a generic element of $A^{<\dN}$, is denoted by $h$.

For two histories $h,h'\in A^{<\dN}$, we write $hh'$ for the concatenation of $h$ with $h'$. Similarly, for a history $h\in A^{<\dN}$ and a run $r\in A^\dN$, we write $hr$ for the concatenation of $h$ with $r$.

Each history induces a subgame of $\Gamma$. Given a history $h\in A^{t-1}$ at stage $t$, the subgame that starts at $h$ is the game $\Gamma^h=(I,(A_i)_{i \in I}, (f_{i,h}, g_{i,h})_{i \in I}, p)$ where $f_{i,h}(h')=f_i(hh')$ for each history $h'\in A^{<\dN}$ and $g_{i,h}(r)=g_i(hr)$ for each run $r\in A^\dN$.

\medskip

\noindent\textbf{Strategies.} A \emph{mixed action} for player $i\in I$ is a probability distribution $x_i$ on her action set $A_i$. The set of all mixed actions for player $i$ is denoted by $\Delta(A_i)$. The \emph{support} of a mixed action $x_i$ is defined as the set of actions on which $x_i$ places positive probability:
$\text{supp}(x_i):=\{a_i\in A_i\colon x_i(a_i)>0\}$.

A \emph{mixed action profile} is a collection of mixed actions $x=(x_i)_{i\in I}$, one for each player. For each $i\in I$, we denote
by $x_{-i}$ the mixed action profile of the opponents of player $i$. A mixed action profile $x$ is called \emph{absorbing} if $p(x)>0$ and \emph{nonabsorbing} if $p(x)=0$, where $p(x)$ is the expectation of $p$ under $x$.

A strategy for player $i\in I$ is a function $\sigma_i:A^{<\dN}\to \Delta(A_i)$.
The interpretation is that if history $h\in A^{<\dN}$ arises, then $\sigma_i$
recommends to choose an action in $A_i$ randomly according to the mixed action $\sigma_i(h)$.

A strategy is called \emph{pure} if, at each history, it places probability one on some action.
A strategy is called \emph{stationary} if it prescribes the same mixed action at all histories.
Thus, we can identify a stationary strategy with a mixed action. Through this identification, we can speak of a support of a stationary strategy.

A \emph{strategy profile} is a collection of strategies $\sigma=(\sigma_i)_{i\in I}$, one for each player.
For each $i\in I$, we denote by $\sigma_{-i}$ the strategy profile of the opponents of player $i$. We can identify each stationary strategy profile with a mixed action profile, and by this identification, we speak of absorbing and nonabsorbing stationary strategy profiles.\medskip

\noindent\textbf{Expected payoffs.}  
Let $\theta$ denote the stage at which the play is absorbed; if the play is never absorbed then $\theta=\infty$. 
Every strategy profile $\sigma$ induces an expected payoff $u(\sigma)=(u_i(\sigma))_{i\in I}$ given by
\[u_i(\sigma)\,:=\, \dP_\sigma(\theta<\infty)\cdot \dE_\sigma(f_i\mid \theta<\infty)+\dP_\sigma(\theta=\infty)\cdot \dE_\sigma(g_i\mid \theta=\infty).\]
Similarly, each strategy profile $\sigma$ induces an expected payoff $u(\sigma\mid h)=(u_i(\sigma\mid h))_{i\in I}$ in each subgame $\Gamma^h$, where $h\in A^{<\dN}$.\medskip

\noindent\textbf{Minmax values.} 
The \emph{minmax value} of player $i\in I$ is defined as
\[v_i\,:=\,\inf_{\sigma_{-i}}\sup_{\sigma_i}u_i(\sigma_i,\sigma_{-i}),\]
where the infimum  is over all strategy profiles of the opponents of player~$i$
and the supremum is over all strategies of player~$i$. Intuitively, the minmax value is the highest expected payoff that player $i$ can defend against her opponents.

In a similar way, we can define the minmax value $v_i(h)$ for each player $i\in I$ in each subgame $\Gamma^h$, where $h\in A^{<\dN}$. At stage 1, for the empty history $\o$, we have $v_i(\o) = v_i$ for each player $i\in I$.\medskip

\noindent\textbf{Absorbing Blackwell games with only two players.} If the game has only two players, then by the definition of the minmax value, for each $\ep> 0$ and each history $h$, player 1 has a strategy $\sigma_1$ such that $u_2(\sigma_1,\sigma_2\mid h)\leq v_2(h)+\ep$ for every strategy $\sigma_2$ for player~2. Such a strategy is called an \emph{$\ep$-punishment strategy} of player 1 at the history $h$. One can define $\ep$-punishment strategies for player 2 similarly.

Moreover, if the game has only two players, then for each player $i\in I=\{1,2\}$ and each history $h$, by Martin (1998) or by Maitra and Sudderth (1998), we have $v_i(h)\,=\,\sup_{\sigma_i}\inf_{\sigma_j}u_i(\sigma_i,\sigma_{j}\mid h)$,
where $j=-i$ is player $i$'s opponent.\medskip

\noindent\textbf{Equilibrium.}  A strategy profile $\sigma$ is an \emph{$\ep$-equilibrium}, where $\ep\geq 0$,
if for each player $i\in I$ and each strategy $\sigma'_i$ of player $i$ we have $u_i(\sigma'_i,\sigma_{-i})\leq u_i(\sigma)+\ep$.
 The definitions imply that
if the strategy profile $\sigma$ is an $\ep$-equilibrium, then $u_i(\sigma)\geq v_i-\ep$ for each player $i\in I$.

\subsection{Classical Absorbing Games}

Two classical classes of absorbing Blackwell games are those in which the payoff functions of all players
are the average of stage-payoffs or the discounted sum of stage-payoffs. In our proof we will use these classical absorbing games as auxiliary tools.

\begin{definition}\label{def-ave-abs}
An \emph{average-payoff absorbing game} is an absorbing Blackwell game $\Gamma = (I,(A_i)_{i \in I}, (f_i, g_i)_{i \in I}, p)$, where for each player $i \in I$ there are functions $r_i:A \to \dR$ (called the absorbing stage-payoff function) and $z_i:A \to \dR$ (called the nonabsorbing stage-payoff function) such that
\begin{equation}\label{g-ave}
f_i(a_1,\ldots,a_t)\,=\,r_i(a_t),\qquad \forall (a_1,\ldots,a_t)\in A^{<\dN},\ \ \forall t\in\dN,
\end{equation}
and 
\begin{equation}\label{f-ave}
g_i(a_1,a_2,\ldots) \,=\, \limsup_{T \to \infty} \frac{1}{T} \sum_{t=1}^T z_i(a_t), \qquad \forall (a_1,a_2,\ldots) \in A^{\dN}.
\end{equation}
\end{definition}

The interpretation of the absorbing payoff function $f_i$ in Eq.~\eqref{g-ave} is that once the play is absorbed through an action profile $a_t$ at some stage $t$, then each player $i$ receives the payoff given by $r_i(a_t)$ at each stage beyond $t$, and therefore her average payoff is $r_i(a_t)$, regardless the stage-payoffs before stage $t$. 

The existence of an $\ep$-equilibrium, for every $\ep > 0$, in two-player and three-player average-payoff absorbing games
was proven by Vrieze and Thuijsman (1989) and Solan (1999), respectively. The question whether average-payoff absorbing games with at least four players always admit an $\ep$-equilibrium, for every $\ep>0$,
is a major open problem in game theory to date.

\begin{definition}\label{def-disc-abs}
Let $\lambda \in (0,1)$.
A \emph{$\lambda$-discounted absorbing game} is an absorbing Blackwell game $\Gamma = (I,(A_i)_{i \in I}, (f_i, g_i)_{i \in I}, p)$, where for each player $i \in I$ there are functions $r_i:A \to \dR$ (called the absorbing stage-payoff function) and $z_i:A \to \dR$ (called the nonabsorbing stage-payoff function) such that
\begin{equation}\label{g-disc}
f_i(a_1,\ldots,a_t) = \sum_{k=1}^{t-1}  \la(1-\lambda)^{k-1} z_i(a_k)+(1-\la)^{t-1}r_i(a_t),\qquad \forall (a_1,\ldots,a_t)\in A^{<\dN},\ \ \forall t\in\dN,
\end{equation}
and
\begin{equation}\label{f-disc}
g_i(a_1,a_2,\ldots) =  \sum_{t=1}^\infty \lambda(1-\lambda)^{t-1} z_i(a_t), \qquad\forall (a_1,a_2,\ldots) \in A^{\dN}.
\end{equation}
We say that a $\lambda$-discounted absorbing game is related to an average-payoff absorbing game if, for each player $i\in I$, the stage-payoff functions $r_i$ and $z_i$ can be chosen the same across the two games.
\end{definition}

The interpretation of the absorbing payoff function $f_i$ in Eq.~\eqref{g-disc} is that once the play is absorbed through an action profile $a_t$ at some stage $t$, then each player $i$ receives the payoff given by $r_i(a_t)$ at each stage beyond $t$, and therefore her total discounted payoff from stage $t$ onwards is $(1-\la)^{t-1}r_i(a_t)$.

It follows from 
Fink (1964) and Takahashi (1964) that every discounted absorbing game admits a stationary 0-equilibrium.

\subsection{Tail-Measurability}\label{Sec.Tail}

A function $\xi:A^\dN\to \dR$ is called \emph{tail-measurable} if whenever two runs $\vec{a}=(a_t)_{t\in \dN}\in A^{\dN}$ and $\vec{a}\hspace{0.5mm}'=(a'_t)_{t\in \dN}\in A^\dN$ satisfy $a_t=a'_t$ for every $t$ sufficiently large, then $\xi(\vec{a})=\xi(\vec{a}\hspace{0.5mm}')$. 
Intuitively, $\xi$ is tail-measurable if the value of $\xi$ is unaffected by changing finitely many coordinates of the run. Not every tail-measurable payoff function is Borel-measurable, see Rosenthal (1975) and Blackwell and Diaconis (1996).




Many payoff functions in the literature of dynamic games are tail-measurable. For example, the average payoff defined in Eq.~\eqref{f-ave}
is tail-measurable. So is the limsup payoff defined by $\xi(a_1,a_2,\ldots) := \limsup_{t\to \infty} z_i(a_t)$, where $z_i:A\to\dR$ is a stage-payoff function.
Various classical winning conditions in the computer science literature are also tail-measurable,
such as the winning conditions B\"uchi, co-B\"uchi, parity, Streett, and M\"uller (see, e.g., Gr\"adel and Ummels (2008), Chatterjee and Henzinger (2012), and Bruy\`{e}re (2021)). The discounted payoff defined in Eq.~\eqref{f-disc} is, however, not tail-measurable.


A property closely related to tail-measurability is shift-invariance  
(or prefix-in\-de\-pendence);
cf.~Chatterjee (2007). 
A payoff function $\xi: A^{\dN} \to \dR$ is called \emph{shift-invariant} if for every run $(a_1,a_2,a_3,\ldots)\in A^\dN$ it holds that:
\begin{equation}\label{eq shift inv}
    \xi(a_1,a_2,a_2,\ldots)=\xi(a_2,a_3,\ldots).
\end{equation}
For example, the average payoff and the limsup payoff are both shift-invariant. As one can verify, every shift-invariant function is also tail-measurable. The converse is not true, as the example in the Introduction illustrated. 

\begin{definition}\label{def-tail-payoff}
Consider an absorbing Blackwell game $\Gamma = (I,(A_i)_{i \in I}, (f_i, g_i)_{i \in I}, p)$. We say that player $i\in I$ has tail-measurable payoffs, if 
the following conditions hold:
\begin{enumerate}
    \item There is a function $r_i:A\to\mathbb{R}$ such that the absorbing payoff function $f_i$ satisfies
\begin{equation}\label{g-shift}
f_i(a_1,\ldots,a_t)=r_i(a_t),\qquad \forall (a_1,\ldots,a_t)\in A^{<\dN},\ \ \forall t\in\dN.
\end{equation}
\item The nonabsorbing payoff function $g_i$ is tail-measurable.
\end{enumerate}
We say that the game $\Gamma$ has tail-measurable payoffs, if each player $i\in I$ has tail-measurable payoffs.
\end{definition}

Intuitively, Condition 1 of Definition \ref{def-tail-payoff} requires that $f_i(a_1,\ldots,a_t)$ is not affected by changing the 
values 
of $(a_1,\ldots,a_{t-1})$. 
According to Definitions \ref{def-ave-abs} and \ref{def-tail-payoff}\black, every average-payoff absorbing game has tail-measurable payoffs.


If the game has tail-measurable 
nonabsorbing payoffs and $r_i:A\to\dR$ is a function as in \Eqref{g-shift}, then for any absorbing mixed action profile $x$, i.e. $p(x)>0$, we define
\[r_i^*(x)\,:=\,\frac{\sum_{a\in A}x(a)p(a)r_i(a)}{\sum_{a\in A}x(a)p(a)},\] where $x(a)$ is the probability of the action profile $a$ under $x$. Thus, $r_i^*(x)$ is the expected absorbing payoff for player $i$ under $x$ conditionally on absorption. Note that, if $p(x)>0$, under the stationary strategy that plays $x$ at each stage, player $i$'s expected payoff is exactly $u_i(x)=r_i^*(x)$.

If player $i$ has tail-measurable payoffs, then provided no absorption has occurred, the payoff functions $g_i$ and $f_i$ do not depend on the actions taken in past stages. This property can be used to show that the minmax value of player $i$ is equal across subgames. Moreover, if there are only two players, i.e., player $i$ has only one opponent, then (i) player $i$ has a mixed action that guarantees that her expected absorbing payoff is 
not below
her minmax value, and similarly, (ii) player $-i$ has a mixed action that guarantees that player $i$'s expected absorbing payoff is 
not above
player $i$'s minmax value. Formally, the following lemma holds.

\begin{lemma}
\label{lemma:independent}
Consider an absorbing Blackwell game $\Gamma$. Assume that player $i\in I$ has tail-measurable payoffs. 
\begin{enumerate}
    \item For every history $h\in A^{<\dN}$, it holds that $v_i(h) = v_i$.
    \item Assume that there are only two players, denoted by $i$ and $-i$. Then, player $i$ has a mixed action $y_i$ with the following property: for every action $a_{-i}\in A_{-i}$ such that $p(y_i,a_{-i})>0$ we have $r^*_i(y_i,a_{-i})\geq v_i$. Also, player $-i$ has a mixed action $y_{-i}$ with the following property: for every action $a_{i}\in A_{i}$ such that $p(a_i,y_{-i})>0$ we have $r^*_i(a_i,y_{-i})\leq v_i$.
\end{enumerate}
\end{lemma}

\begin{proof}\\
\noindent\textbf{Proof of Part 1.}  Take an arbitrary history $h\in A^{t-1}$ at some stage $t$. We show that $v_i(h)=v_i$. 

As the actions before stage $t$ do not influence $f_i$, we have $v_i(h)=v_i(h')$ for all $h'\in A^{t-1}$.  Fix an $\ep > 0$.\smallskip

\noindent\textsc{Step 1:} We show that $v_i(h)\geq v_i-\ep$. We distinguish between two cases.\smallskip

\noindent\textsc{Case 1:} For some $\rho > 0$, against every mixed action profile $x_{-i}$, player $i$ has an action $a_i \in A_i$ such that $p(a_i,x_{-i}) \geq \rho$ and $r^*_i(a_i,x_{-i}) \geq v_i - \ep$.

In this case, take an arbitrary strategy profile $\sigma_{-i}$ of the opponents of player $i$. Let $\sigma_i$ be the following strategy of player~$i$: at any history $h'$, play an action $a_i$ that satisfies $p(a_i,\sigma_{-i}(h')) \geq \rho$ and $r^*_i(a_i,\sigma_{-i}(h')) \geq v_i - \ep$.
Then $\prob_{\sigma_i,\sigma_{-i}}(\theta < \infty \mid h) = 1$ and $u_i(\sigma_i,\sigma_{-i} \mid h) \geq v_i - \ep$. Hence $v_i(h)\geq v_i-\ep$.\smallskip

\noindent\textsc{Case 2:} Case 1 does not hold.

Then, for every $\rho > 0$, there is a mixed action profile $x_{-i}^\rho$ such that for every $a_i \in A_i$ we have $p(a_i,x^\rho_{-i}) < \rho$ or $r^*_i(a_i,x^\rho_{-i}) < v_i - \ep$.
Let $x_{-i}^0$ be an accumulation point of $(x_{-i}^\rho)_{\rho>0}$ as $\rho$ goes to 0. Then, for every $a_i \in A_i$ we have $p(a_i,x^0_{-i}) = 0$ or $r^*_i(a_i,x^0_{-i}) \leq v_i - \ep$.

Let $\sigma'_{-i}$ be a strategy profile such that $u_i(\sigma_i,\sigma'_{-i}\mid h')\leq v_i(h')+\ep=v_i(h)+\ep$ for each history $h' \in A^{t-1}$ and each strategy $\sigma_i$ of player $i$. Let $\sigma_{-i}$ be the strategy profile that plays the mixed action profile $x^0_{-i}$ before stage $t$, and if no absorption occurs before stage $t$, then at stage $t$ switches to $\sigma'_{-i}$. 
By construction, $u_i(\sigma_i,\sigma_{-i}) \leq \max\{v_i-\ep,v_i(h)+\ep\}$  for each strategy $\sigma_i$. By the definition of $v_i$, we 
have
$v_i\leq v_i(h)+\ep$.\smallskip

\noindent\textsc{Step 2:} We show that $v_i\geq v_i(h)-\ep$. Take an arbitrary strategy profile $\sigma_{-i}$. It suffices to prove that there is a strategy $\sigma_i$ such that $u_i(\sigma_i,\sigma_{-i}) \geq v_i(h)- \ep$.

By the definition of $v_i(h)$, for every mixed action profile $x_{-i}$ there is an action $a_i$ such that either $p(a_i,x_{-i})=0$ or $r^*_i(a_i,x_{-i})\geq v_i(h)$; otherwise player $i$ could not defend $v_i(h)$ in the subgame at $h$ if his opponents are playing $x_{-i}$ at each stage.

Let $\sigma'_i$ be a strategy such that $u_i(\sigma'_i,\sigma_{-i}\mid h') \geq v_i(h')- \ep=v_i(h)- \ep$ for each history $h'\in A^{t-1}$.
Let $\sigma_i$ be the strategy that at any history $h'$ before stage $t$, plays an action $a_i$ such that either $p(a_i,\sigma_{-i}(h'))=0$ or $r^*_i(a_i,\sigma_{-i}(h'))\geq v_i(h)$, and if no absorption occurs before stage $t$, then at stage $t$ switches to $\sigma'_i$. Then, by construction, $u_i(\sigma_i,\sigma_{-i}) \geq \min\{v_i(h),v_i(h)-\ep\}=v_i(h)-\ep$.\medskip

\noindent\textbf{Proof of Part 2.} Assume that player $i$ has only one opponent. Consider the two-player zero-sum one-shot game in which: (i) the action sets are $A_i$ and $A_{-i}$, (ii) if action pair $a\in A$ is chosen, then the payoff is equal to $p(a)r_i(a)+(1-p(a))v_i$, and (iii) player $i$ maximizes the payoff and player $-i$ minimizes it. 
Intuitively, the payoff is equal to the expected minmax value of player $i$ in the continuation 
subgame; indeed, the play absorbs with probability $p(a)$ and then the absorbing payoff is $r_i(a)$ for player $i$, and the game moves to the next stage with probability $1-p(a)$ and then, by Part 1, player $i$'s continuation minmax value is $v_i$. 

By the dynamic programming principle, the value of this one-shot game is equal to $v_i$. 
Let $y_i$ be any mixed action of player~$i$ that is optimal in this one-shot game. We show that $y_i$ satisfies the property in Part 2 of the lemma. To this end, let $a_{-i} \in A_{-i}$ with $p(y_i,a_{-i})>0$. By the choice of $y_i$, 
\[\sum_{a_i\in A_i}y_i(a_i)\cdot\Big(p(a_i,a_{-i})r_i(a_i,a_{-i})+(1-p(a_i,a_{-i}))v_i\Big)\,\geq\,v_i,\]
yielding 
\[r_i^*(y_i,a_{-i})\,=\,\frac{\sum_{a_i\in A_i}y_i(a_i)p(a_i,a_{-i})r_i(a_i,a_{-i})}{\sum_{a_i\in A_i}y_i(a_i)p(a_i,a_{-i})}\,\geq\,v_i.\]

Similarly, if $y_{-i}$ is any mixed action of player~$-i$ that is optimal in this one-shot game, then $y_{-i}$ satisfies the property in Part 2 of the lemma: for every action $a_{i} \in A_{i}$ such that $p(a_i,y_{-i})>0$ we have $r^*_i(a_i,y_{-i}) \leq v_i$.
\end{proof}

\begin{remark}\rm 
In Part 1 of Lemma \ref{lemma:independent}, we showed that tail-measurability implies that player $i$'s minmax value $v_i(h)$ is the same for each history $h$. This is no longer true if there are multiple nonabsorbing states. 
In that case, under tail-measurability, the player's minmax value $v_i(h)$ can also depend on the current stage of the history~$h$. For an illustration, we refer to Example 4.1 in Flesch and Solan (2022). 
\end{remark}

\section{Main Theorem}\label{sec-mainres}

Now we present the main 
result
of the paper.

\begin{theorem}
\label{theorem:1}
Every two-player absorbing Blackwell game with tail-measurable payoffs admits an $\ep$-equilibrium, for every $\ep > 0$.
\end{theorem}

A related result is proven in Ashkenazi-Golan, Flesch, Predtetchinski, and Solan~(2022). They studied multi-player Blackwell games with no absorption, i.e., when $p(a) = 0$ for every action profile $a \in A$, and showed that these games admit an $\ep$-equilibrium for every $\ep > 0$, provided that the payoffs are tail-measurable (since the absorption probabilities are zero, the
functions 
$(f_i)_{i \in I}$ play no role). In comparison, our main result allows some of the entries to be absorbing with positive probability,
yet it is restricted to two-player games.

Another strongly related result is proven in Flesch and Solan (2022). They studied two-player stochastic games with finite state and action spaces, and showed that these games admit an $\ep$-equilibrium for every $\ep > 0$, provided that the payoffs are shift-invariant. In comparison, our main result uses the weaker condition of tail-measurability on the payoffs, yet it allows only one nonabsorbing state. 

The idea of the proof of Theorem \ref{theorem:1} is as follows. We define an auxiliary average-payoff absorbing game (cf. Section \ref{sec-aux}). By using the method and results of Vrieze and Thuijsman (1989), we can categorize this auxiliary game into three essentially different cases. In the first two cases (cf.~Sections \ref{sec-case1} and \ref{sec-case2}), the $\ep$-equilibrium constructed by Vrieze and Thuijsman for the auxiliary game leads us to a similar $\ep$-equilibrium in the absorbing Blackwell game. In the third case (cf.~Section \ref{sec-case3}), we find an $\ep$-equilibrium by using arguments from Flesch and Solan (2022). \medskip

To illustrate some of the ideas of the proof,
we study an example, which is a generalization of the famous Big Match, cf. Gillette (1957), Blackwell and Ferguson (1968).

\begin{example}\label{ex bm}
\rm Consider the following two-player absorbing Blackwell game $\Gamma$. Each player has two actions: the actions of player~1 are $C$ (continue) and $Q$ (quit), and the actions of player~2 are $L$ (left) and $R$ (right). As long as player~1 plays $C$ the play continues, and as soon as player~1 plays $Q$ the play is absorbed. Formally, $A_1=\{C,Q\}$, $A_2=\{L,R\}$, $p(C,L)=p(C,R)=0$ and $p(Q,L)=p(Q,R)=1$.

The absorbing payoffs are as follows: $(1,0)$ if the play is absorbed under $(Q,L)$, and $(0,1)$ if the play is absorbed under $(Q,R)$. Formally, denoting $f=(f_1,f_2)$, we have $f(a_1,\ldots,a_t)=(1,0)$ if $a_t=(Q,L)$ and $f(a_1,\ldots,a_t)=(0,1)$ if $a_t=(Q,R)$
(see Figure~\arabic{figurecounter}, where absorbing payoffs are denoted by $\star$).
\[
\begin{array}{c||c|c|}
\text{Game }\Gamma & L & R\\
\hline\hline
C & & \\
\hline
Q & 1,0\ ^\star & 0,1 \ ^\star\\
\hline
\end{array}\vspace{0.2cm}
\]
\centerline{Figure~\arabic{figurecounter}: The game in Example~\ref{ex bm}.}
\addtocounter{figurecounter}{1}
\medskip

If player~1 always continues, the nonabsorbing payoff function for each player $i\in\{1,2\}$ is some bounded, Borel-measurable and tail-measurable function $g_i$. Let $g=(g_1,g_2)$, and denote by $R_C$ the set of all runs in which player~1 always plays $C$.

Since by quitting player~1 guarantees a payoff at least 0 to herself and at most 1 to player~2, we have $v_1 \geq 0$ and $v_2 \leq 1$. We distinguish 3 cases.

\textsc{Case 1:} $v_1 > 1$. Then, we necessarily have $g_1(r) > 1$ for every run $r \in R_C$.
One $\ep$-equilibrium is the one in which both players follow a run $r^*\in R_C$ such that $g_2(r^*)\geq \sup_{r \in R_C} g_2(r)-\ep$. (This will correspond to Case 3 of Theorem \ref{TV-thm}, see Proposition \ref{propcase31}.)

\textsc{Case 2:} $v_1\leq 1$ and $v_1 + v_2 \leq 1$. Consider the stationary strategy $x_1=(1-\frac{\ep}{2})C + \frac{\ep}{2} Q$ for player~1, and the stationary strategy $x_2=(1-\alpha)L + \alpha R$ for player~2, where $\alpha := \max\{0, v_2\}$. Note that $(x_1,x_2)$ is absorbing, and it gives expected payoff $u_1(x_1,x_2)=1-\alpha=\min\{1,1-v_2\}\geq v_1$ and $u_2(x_1,x_2)=\alpha\geq v_2$.

Then, one $\ep$-equilibrium is the one in which player~1 plays $x_1$ and player~2 plays $x_2$, supplemented with standard statistical tests to monitor the other player for deviations:
if the realized action-frequency of player~2 is not close to $(1-\eta)L + \eta R$ (by using the law of large numbers),
player~2 is punished at her minmax level;
if sufficiently many stages have passed without player~1 selecting $Q$, 
player~1 is punished at her minmax level. (This will correspond to Case 2 of Theorem \ref{TV-thm}, see Proposition \ref{lemma eq2}.)

\textsc{Case 3:} $v_1\leq 1$ and $v_1 + v_2 > 1$. In this case, $v_1+v_2>1$ implies that there is no stationary strategy pair that gives satisfactory absorbing payoffs to both players, and this is a difficult case. (This will belong to Case 3 of Theorem \ref{TV-thm}, see Proposition \ref{propcase31}.)

We explain the idea in a brief manner. In the arguments below, $\delta>0$ should be thought of as a quantity small compared to $\ep$ and satisfying $v_1 + v_2 \geq 1 + 2\delta$.
Let $\sigma_2$ be a subgame $\delta$-minmax strategy of player~2;
that is, a strategy that guarantees to player~2 at least $ v_2 - \delta$ in all subgames.
If player~1 quits at history $h$,
player~2's payoff is $\sigma_2(R \mid h)$.
Since $\sigma_2$ is subgame $\delta$-minmax,
it follows that $\sigma_2(R \mid h) \geq v_2 - \delta$, for every $h \in A^{<\mathbb{N}}$.
The condition $v_1 + v_2 \geq 1 + 2\delta$
implies that for every $h \in A^{<\mathbb{N}}$
\[ \sigma_2(L \mid h) \,=\, 1-\sigma_2(R \mid h) 
\,\leq\, 1 - v_2 + \delta \,\leq\,v_1 - \delta. \]
This implies that player 1 by quitting cannot get more than $v_1-\delta$. Hence, player 1's best response to $\sigma_2$ is the strategy $\sigma_1^C$ that continues forever. Let $w \in \dR^2$ be a vector such that 
$\prob_{\sigma_1^C,\sigma_2}\bigl(\|g(r) - w\|_\infty \leq \delta\bigr) > 0$, where $r$ denotes the random variable for the run. This means that the payoff under $(\sigma_1^C,\sigma_2)$ is close to $w$ with positive probability. 
Since the payoffs are bounded and Borel-measurable,
L\'evy's zero-one law implies that there is a history $h$ such that 
$\prob_{\sigma_1^C,\sigma_2}\bigl(\|g(r) - w\|_\infty \leq \delta \mid h\bigr) \geq 1-\delta$. If $\delta>0$ is sufficiently small, by the choice of $\sigma_2$ and because $\sigma_1^C$ is a best response, $w_1\geq v_1-\ep$ and $w_2\geq v_2-\ep$.

Since all probability measures on $A^\dN$ are regular,
there is a closed set $ D\subseteq A^\dN$ of runs such that 
$\|g(r) - w\|_\infty \leq \delta$ for every $r\in D$ and
$\prob_{\sigma_1^C,\sigma_2}(D) \geq 1-2\delta$. Since $D$ is closed, its complement is open.
It follows that for every run $r \not\in D$
there is a stage $k \in \dN$ such that all runs in $A^\dN$ that coincide with $r$ in the first $k$ stages are also not in $D$.
In other words, at stage $k$ it is known that the run is not in $D$.
A $2\ep$-equilibrium suggests itself:
in the first $\ell$ stages, where $\ell$ is the length of $h$,
the players play $(C,R)$.
Afterwards the players follow $(\sigma_1^C,\sigma_2)$,
as if the history at stage $\ell$ was $h$,
until we reach some stage $k$ such that at stage $k$ it is known that the run is not in $D$.
From stage $k+1$ and on player 1 punishes player 2 at her minmax level. $\Diamond$
\end{example}

For the remaining part of the paper, fix a two-player absorbing Blackwell game $\Gamma$ with tail-measurable payoffs. 
For each player $i\in\{1,2\}$, 
let $r_i$ be
the function given by Definition \ref{def-tail-payoff}. Fix also an arbitrary $\ep>0$. We will show that $\Gamma$ admits a $2\ep$-equilibrium.

\section{The Auxiliary Game}\label{sec-aux}

Since the absorbing Blackwell game $\Gamma$ has tail-measurable payoffs, by Part 1 of Lemma \ref{lemma:independent}, the minmax values do not change as long as no absorption occurs. \smallskip

We derive an auxiliary average-payoff absorbing game from $\Gamma$ by setting the nonabsorbing payoff for player~1 to be the constant $v_1-\ep$, and for player~2 to be the constant $v_2-\ep$. Formally, this auxiliary game is defined as follows.

\begin{definition}
Let $G^\infty_{v-\ep}$ be the average-payoff absorbing game $(I,(A_i)_{i \in I}, (f_i, g^\ep_i)_{i \in I}, p)$
where $g^\ep_i(a_1,a_2,\ldots) = v_i-\ep$ for each player $i \in I$ and each%
\footnote{Thus, following Definition \ref{def-ave-abs}, in $G^\infty_{v-\ep}$ the nonabsorbing stage-payoff function $z_i$ is given by $z_i(a)=v_i-\ep$ for each player $i\in I$ and each action profile $a\in A$.}
run $(a_1,a_2,\ldots)\in A^\dN$.
\end{definition}

 Note that if a stationary strategy pair $x$ is absorbing, then $x$ induces the same expected payoff in the auxiliary game $G^\infty_{v-\ep}$ and in the original game $\Gamma$.

Let $v_1^\infty$ and $v_2^\infty$ denote the minmax values of the players in the game $G^\infty_{v-\ep}$. 
As we now show, 
$v_1^\infty$ and $v_2^\infty$
are close to the minmax values in the original game $\Gamma$.

\begin{lemma}\label{new-minmax}
$v_1-\ep\leq v_1^\infty\leq v_1$ and $v_2-\ep\leq v_2^\infty\leq v_2$.
\end{lemma}

\begin{proof}
We only prove $v_1-\ep\leq v_1^\infty\leq v_1$; the proof is similar for player~2.

First we prove that $v_1-\ep\leq v_1^\infty$. In view of Part 2 of Lemma \ref{lemma:independent}, in the game $\Gamma$, player 1 has a mixed action $y_1$ with the following property: for any action $a_2 \in A_2$ for which $(y_1,a_2)$ is absorbing we have $r^*_1(y_1,a_2) \geq v_1$. 
In the game $G^\infty_{v-\ep}$, if player 1 always plays $y_1$, player~1 obtains an expected payoff of at least $v_1 - \ep$.
Thus, $v_1-\ep\leq v_1^\infty$.

We now prove that $v_1^\infty \leq v_1$. In view of Part 2 of Lemma \ref{lemma:independent}, in the game $\Gamma$, player 2 has a mixed action $y_2$ with the following property: for any action $a_1 \in A_1$ for which $(a_1,y_2)$ is absorbing we have $r^*_1(a_1,y_2) \leq v_1$. 
In the game $G^\infty_{v-\ep}$, if player $2$ always plays $y_2$, player~1 cannot get an expected payoff higher than $v_1$.
Thus $v_1^\infty \leq v_1$.
\end{proof}\bigskip

We apply the method and results in Vrieze and Thuijsman (1989) and Thuijsman (1992) to the game $G^\infty_{v-\ep}$. To this end, for each discount factor $\la\in(0,1)$, we also consider the $\la$-discounted absorbing game $G^\la_{v-\ep}$ related to $G^\infty_{v-\ep}$ (i.e., these games have the same stage-payoff functions, cf. Definition \ref{def-disc-abs}). Let $x^\la$ be a stationary $\la$-discounted equilibrium in $G^\la_{v-\ep}$,
for each $\la\in(0,1)$. As the space of stationary strategies is compact,\footnote{Or, alternatively, by using the theory of semi-algebraic sets.}
there is a sequence $(\la_n)_{n=1}^\infty$ converging to 0 such that $x^{\la_n}$ converges to some stationary strategy pair $x^0$ as $n\to\infty$. 
\smallskip

Since player~$i$'s nonabsorbing payoff in the game $G^\infty_{v-\ep}$ is $v_i-\ep$, the study of Vrieze and Thuijsman (1989) implies that $G^\infty_{v-\ep}$ falls into at least one 
of three (not exclusive) categories.\footnote{The category to which $G^\infty_{v-\ep}$ belongs depends on the choices of $x^\lambda$ and $x^0$. In Theorem \ref{TV-thm}, Cases 1, 2, and 3 correspond to respectively to Cases A, C, and B in Vrieze and Thuijsman (1989).}

\begin{theorem}[Vrieze and Thuijsman (1989), Thuijsman (1992)]\label{TV-thm}\
At least one of the following three conditions holds for the game $G^\infty_{v-\ep}$:
\begin{enumerate}
\item  The following holds:
\begin{enumerate}
    \item $x^0$ is absorbing: $p(x^0)>0$.
    \item  The payoff under $x^0$ is at least the minmax value in $G^\infty_{v-\ep}$: for each player $i\in\{1,2\}$ we have $r^*_i(x^0) \geq v_i^\infty$. Consequently, by Lemma \ref{new-minmax}, for each player $i\in\{1,2\}$ we have
$r^*_i(x^0) \geq v_i-\ep$.
\item   There is no profitable absorbing deviation from $x^0$:
for each player $i\in\{1,2\}$ and each action $a_i \in A_i$ such that $p(a_i,x^0_{-i}) > 0$, we have $r^*_i(a_i,x^0_{-i}) \leq r^*_i(x^0)$. 
\end{enumerate}

\item There is a player $i\in\{1,2\}$ and an action $\widehat a_i \in A_i$ such that $x^0$, $i$, and $\widehat a_i$ satisfy:
\begin{enumerate}
\item   $x^0$ is nonabsorbing: $p(x^0) = 0$.
\item   $(\widehat a_i,x^0_{-i})$ is absorbing: $p(\widehat a_i,x^0_{-i}) > 0$.
\item   The payoff under $(\widehat a_i,x^0_{-i})$ is at least the minmax value  in $G^\infty_{v-\ep}$: for each player $j=1,2$ we have $r^*_j(\widehat a_i,x^0_{-i}) \geq v_j^\infty$. Consequently, by Lemma \ref{new-minmax}, for each player $j\in\{1,2\}$ we have
$r^*_j(\widehat a_i,x^0_{-i}) \geq v_j-\ep$.
\item   The payoff under $(\widehat a_i,x^0_{-i})$ is at least as good as any absorbing deviation from $x^0$:
for each player $j=1,2$ and each action $a_j \in A_j$ such that $p(a_j,x^0_{-j}) > 0$, we have $r^*_j(a_j,x^0_{-j}) \leq r^*_j(\widehat a_i,x^0_{-i})$.
\end{enumerate}

\item   The following holds:
\begin{enumerate}
\item   $x^0$ is nonabsorbing: $p(x^0) = 0$.
\item   There is no profitable absorbing deviation from $x^0$:
for each player $i\in\{1,2\}$ and each action $a_i \in A_i$ such that $p(a_i,x_{-i}^0) > 0$, we have $r^*_i(a_i,x_{-i}^0) \leq v_i-
\ep$.
\end{enumerate}
\end{enumerate}
\end{theorem}

Our proof for the existence of an $\ep$-equilibrium in the absorbing Blackwell game $\Gamma$ will distinguish between three cases, depending on which case the auxiliary game $G^\infty_{v-\ep}$ falls into.
As we will see,
the structure of the $\ep$-equilibrium in $\Gamma$ when Case~1 or Case~2 of Theorem~\ref{TV-thm} holds, is similar to the corresponding $\ep$-equilibrium in Vrieze and Thuijsman (1989).
The challenging part is Case~3 of Theorem~\ref{TV-thm},
where the $\ep$-equilibrium in Vrieze and Thuijsman (1989) is nonabsorbing. 

\section{The Proof of Theorem \ref{theorem:1}}\label{sec-proof}

Recall that we fixed a two-player absorbing Blackwell game $\Gamma$ with tail-measurable payoffs, and an arbitrary $\ep>0$. 
We will show that $\Gamma$ admits a $2\ep$-equilibrium.

Let $M$ be a bound on the payoffs in the absorbing Blackwell game $\Gamma$:
\[M\,:=\,\max_{i\in\{1,2\}}\max\{\sup_{h\in A^{<\dN}}|g_i(h)|,\sup_{r\in A^\dN}|f_i(r)|\}.\]

\subsection{Case 1 of Vrieze and Thuijsman}\label{sec-case1}

In this section we consider Case 1 of Theorem \ref{TV-thm}, where the limit strategy profile $x^0$ is absorbing.

\begin{proposition}\label{lemma eq1}
Assume that Case 1 of Theorem \ref{TV-thm} holds. Then, the absorbing Blackwell game $\Gamma$ admits a $2\varepsilon$-equilibrium.
\end{proposition}
\begin{proof} Our arguments are close to those of Vrieze and Thuijsman (1989), and therefore we only provide a sketch. 



By Condition 1(b) of Theorem \ref{TV-thm}, there is a $\eta>0$ sufficiently small so  that $\eta\cdot 2M\leq \ep$ and
\begin{equation}
\label{choice-delta}
(1-\eta)\cdot r_i^*(x^0)+\eta\cdot M\,\geq\,v_i-\frac{3\ep}{2}.
\end{equation}
Let $N\in\mathbb{N}$ satisfy
\begin{equation}
\label{choice-N}
1-(1-p(x^0))^N\,\geq\, 1-\eta.
\end{equation}

The game $\Gamma$ admits an $2\ep$-equilibrium $\sigma^*$ of the following form:  the players play the mixed action profile $x^0$ for the first $N$ stages, which by \Eqref{choice-N} leads to absorption with probability at least $1-\eta$,
and if no absorption has taken place by stage $N$, then each player switches to an $\frac{\ep}{2}$-punishment strategy against her opponent.%
\footnote{The result of Proposition \ref{lemma eq1} also holds for games with more than two players, but then punishment must be directed against a specific player. In order to identify the player who deviated from $x$, the players conduct statistical tests on the realized action frequencies, see Solan (1999, Lemma 5.2).}
Let us verify that $\sigma^*$ is a $2\ep$-equilibrium in $\Gamma$.
By \Eqref{choice-delta}, 
the expected payoff under $\sigma^*$ is at least $v_i - \frac{3\ep}{2}$.
Because of Condition~1(c) of Theorem \ref{TV-thm} and the punishment strategies after stage $N$, an absorbing deviation from $\sigma^*$ 
can increase the deviator's payoff by at most $\eta\cdot 2M\leq \ep$, whereas 
a nonabsorbing deviation from $\sigma^*$ can increase the deviator's payoff by at most $2\ep$.
\end{proof}

\begin{example}\label{ex-abs}\rm \textbf{(An absorbing equilibrium)} Consider the following 
variation of Example~\ref{ex bm}. The action sets and absorbing payoffs are as in Example~\ref{ex bm}.

If player~1 always continues, the nonabsorbing payoff function $f_1$ for player~1 gives payoff 1 if $L$ is played only finitely many times, and payoff 0 otherwise. For player~2, $f_2$ is an arbitrary tail-measurable function.


The minmax value of player~1 in $\Gamma$ is $v_1=0$, because player~1's expected payoff is at most $\delta$ if player~2 plays $L$ with probability $\delta$ and $R$ with probability $1-\delta$ at each stage, where $\delta\in(0,1)$.

Take an arbitrary $\ep>0$. The auxiliary absorbing game $G^\infty_{v-\ep}$, has stage-payoffs $(v_1-\ep,v_2-\ep)=(-\ep,v_2-\ep)$ under both $(C,L)$ and $(C,R)$. The game $G^\infty_{v-\ep}$ can be represented as follows:
\[
\begin{array}{c||c|c|}
\text{Game }G^\infty_{v-\ep} & L & R\\
\hline\hline
C & (-\ep,v_2-\ep) & (-\ep,v_2-\ep)\\
\hline
Q & 1,0 \ ^\star& 0,1\ ^\star\\
\hline
\end{array}\vspace{0.2cm}
\]

The minmax value of player~1 in $G^\infty_{v-\ep}$ is $v_1^\infty=0$. Note that $v_i-\ep\leq v_i^\infty\leq v_i$ for each player $i\in\{1,2\}$, in accordance with Lemma \ref{new-minmax}.

For every discount factor $\la>0$, in the game $G^\la_{v-\ep}$, playing $(Q,R)$ at each stage is a stationary $\la$-discounted equilibrium, so we can take $x^\la=x^0=(Q,R)$. Thus, $x^0$ is absorbing (cf.~Case 1 of Theorem \ref{TV-thm}).

We obtain the following $\ep$-equilibrium for the original game $\Gamma$, for $\ep\in(0,1)$: At stage 1, the players are supposed to play $(Q,R)$, yielding absorption with payoff $(0,1)$. If player~1 deviates to $C$, then at stage 2 player~2 switches to an $\ep$-punishment strategy: at each further stage, player~2 plays $L$ with probability $\ep$ and $R$ with probability $1-\ep$, which makes sure that player~1's payoff is at most $\ep$. $\Diamond$
\end{example}

\subsection{Case 2 of Vrieze and Thuijsman}\label{sec-case2}

In this section we consider Case 2 of Theorem \ref{TV-thm}.

\begin{proposition}\label{lemma eq2}
Assume that Case 2 of Theorem \ref{TV-thm} holds. Then, the absorbing Blackwell game $\Gamma$ admits an $2\varepsilon$-equilibrium.
\end{proposition}

\begin{proof}
As in Case~1,
our arguments are very close to those of Vrieze and Thuijsman (1989), and therefore we only provide a sketch
of the proof. 

Let $\delta\in(0,1)$. By Condition~2(c) of Theorem \ref{TV-thm}, there is a $\eta>0$ sufficiently small so  that $\eta\cdot 2M\leq \ep$ and
\begin{equation}
\label{choice-delta2}
(1-\eta)\cdot r_i^*((1-\delta)x^0_i+\delta \widehat{a}_i, x^0_{-i})+\eta\cdot M\,\geq\,v_i-\frac{3\ep}{2}.
\end{equation}
Then, let $N\in\mathbb{N}$ so that
\begin{equation}
\label{choice-N2}
1-(1-p((1-\delta)x^0_i+\delta \widehat{a}_i, x^0_{-i}))^N\,\geq\, 1-\eta.
\end{equation}

If $\delta>0$ is sufficiently small, the game $\Gamma$ admits an $\ep$-equilibrium $\sigma^*$ of the following form: the players play the mixed action profile $\widehat{x}=\left((1-\delta)x^0_i+\delta \widehat{a}_i, x^0_{-i}\right)$ for the first $N$ stages, which by \Eqref{choice-N2} leads to absorption with probability at least $1-\eta$. 
The players switch to $\frac{\ep}{2}$-punishment strategies if one of the following events occurs: (i) no absorption takes place by stage $N$, or (ii) during the play, a statistical test (based on the law of large numbers) detects a deviation by player $-i$: the frequency with which player $-i$ chooses her actions is not sufficiently close to $x^0_{-i}$, including the event that player $-i$ chooses an action that has probability zero under $x^0_{-i}$. 
Provided that $\delta$ is sufficiently small
and $N$ is sufficiently large,
the probability of false detection 
of deviation
is small.\footnote{Examples of such statistical test can be found in the proof of Lemma~5.1 in Solan (1999). }

Note that by \Eqref{choice-delta2}, 
the expected payoff under $\sigma^*$ is at least $v_i - \frac{3\ep}{2}$. There is no individual deviation that would increase the payoff of the deviator by more than $\ep$. Indeed, (a) player $i$ cannot gain more than $\eta\cdot 2M\leq\ep$ by an absorbing deviation due to Condition 2(d) of Theorem \ref{TV-thm};
(b) nonabsorbing deviations of player~$i$ trigger punishment after stage $N$,
and hence it can lead to a gain of at most $\frac{3\ep}{2}+\frac{\ep}{2}=2\ep$;
(c) deviations of 
player $-i$ within the support of $x^0_{-i}$ 
are either detected statistically 
and punished,
or do not affect by much the probability to be absorbed in $N$ stages
and the expected absorbing payoff;
and (d) a deviation of 
player $-i$ outside the support of $x^0_{-i}$ is followed by a punishment, and punishment is effective due to Condition~2(d) of Theorem \ref{TV-thm}
and since $\delta$ is small.
\end{proof}

\subsection{Case 3 of Vrieze and Thuijsman}\label{sec-case3}

In this section we consider Case 3 of Theorem \ref{TV-thm}.

Let $Y_1$ be the set of mixed actions $y_1$ for player 1 with the following property: for every action $a_2\in A_2$ such that $p(y_1,a_2)>0$ we have $r^*_1(y_1,a_2)\geq v_1$. Similarly, let $Y_2$ be the set of mixed actions $y_2$ for player 2 with the following property: for every action $a_1\in A_1$ such that $p(a_1,y_2)>0$ we have $r^*_2(a_1,y_2)\geq v_2$. In view of Part 2 of Lemma~\ref{lemma:independent}, these sets are nonempty. One can view these mixed actions as safe, as they do not allow bad expected absorbing payoffs for the players.

Theorem \ref{difficultcase} below follows from Lemma 3.29 in Flesch and Solan (2022) (which is closely related to Proposition 32 in Vieille (2000a), called the
\emph{property of the alternatives}). The proof of Lemma 3.29 in Flesch and Solan (2022) is complicated, but could be somewhat simplified in our case, since absorbing Blackwell games have only one nonabsorbing state. Yet, in this part of the proof of Theorem \ref{theorem:1}, the simplification would be much less significant, and it is less clear how to use classical techniques from absorbing games
to obtain it.
Therefore, we do not repeat the proof of Flesch and Solan (2022) and only state the 
result. 

\begin{theorem}[Flesch and Solan (2022)]\label{difficultcase}
Assume the following for the absorbing Blackwell game $\Gamma$:
\begin{enumerate}
    \item All combinations of mixed actions in $Y$ are nonabsorbing: $p(y_1,y_2)=0$ for every $(y_1,y_2)\in Y_1\times Y_2$. 
    \item Player 2 has no good absorbing response to any mixed action in $Y_1$: for every $y_1\in Y_1$ and $a_2\in A_2$ such that $p(y_1,a_2)>0$ we have $r^*_2(y_1,a_2)<v_2$.
    \item Player 1 has no good absorbing response to any mixed action in $Y_2$: for every $y_2\in Y_2$ and $a_1\in A_1$ such that $p(a_1,y_2)>0$ we have $r^*_1(a_1,y_2)<v_1$.
\end{enumerate} 
Then, for every $\ep>0$, there is a history $h$ such that the subgame that starts at $h$ admits an $\ep$-equilibrium.\footnote{The required history is the history $h^*$ in Step 3.1 of the proof of Lemma 3.29 in Flesch and Solan (2022).}
\end{theorem}

We show that $\Gamma$ has an $\ep$-equilibrium when Case 3 of Theorem \ref{TV-thm} holds, by distinguishing between two cases depending on whether the conditions of Theorem \ref{difficultcase} hold. 

\begin{proposition}\label{propcase31} Assume that Case 3 of Theorem \ref{TV-thm} holds and the condition of Theorem \ref{difficultcase} also holds. Then, the absorbing Blackwell game $\Gamma$ admits an $\ep$-equilibrium.
\end{proposition}

\begin{proof}
By Theorem \ref{difficultcase}, there is a history $h$ such that the subgame that starts at $h$ admits an $\ep$-equilibrium $\sigma$. Assume that $h$ is a history at stage $t$, i.e., $h\in A^{t-1}$.

The following strategy profile $\sigma^*$ is an $\ep$-equilibrium: each player $i\in\{1,2\}$ plays the mixed action $x^0_i$ in stages $1,2,\ldots,t-1$, and in the remaining game, regardless of the realized history at stage $t$, the players play according to $\sigma$ as if the history $h$ was realized.%
\footnote{An alternative $\ep$-equilibrium would be to let the players play from stage $t$ onward according to $\sigma$, as if the history $h$ was realized, and then use backward induction for the stages $1,2,\ldots,t-1$.}

Note that playing from stage $t$ onward as if the history $h$ was realized can be done without affecting the payoffs because the payoff functions are tail-measurable. In particular, each player $i$'s expected payoff under $\sigma^*$ is at least her minmax value minus $\ep$, i.e., $u_i(\sigma^*)\geq v_i-\ep$. Note that neither player can deviate profitably from $\sigma^*$ in the first $t-1$ stages due to Condition~3(b) of Theorem \ref{TV-thm}, and from stage $t$ onward, neither player can gain more than $\ep$ by an individual deviation from $\sigma^*$ due to the choice of $\sigma$.
\end{proof}

\begin{proposition} Assume that Case 3 of Theorem \ref{TV-thm} holds, 
yet
the condition of Theorem \ref{difficultcase} does not hold. Then, there is a player $i\in\{1,2\}$, an action $a_i\in A_i$, and a mixed action $y_{-i}\in Y_{-i}$ such that:
\begin{itemize}
    \item[(a)] $(a_i,y_{-i})$ is absorbing: $p(a_i,y_{-i})>0$.
    \item[(b)] Each player's expected absorbing payoff under $(a_i,y_{-i})$ is at least her minmax value: $r^*_j(a_i,y_{-i})\geq v_j$ for each $j\in\{1,2\}$.
    \item[(c)] Player $i$ has no profitable absorbing deviation: for every $a'_i\in A_i$ such that $p(a'_i,y_{-i})>0$ we have $r^*_i(a_i,y_{-i})\geq r^*_i(a'_i,y_{-i})$.
\end{itemize}
Consequently, the absorbing Blackwell game $\Gamma$ admits an $\ep$-equilibrium.
\end{proposition}

\begin{proof}\\
\noindent\textbf{Part 1.} First we show that there is a player $i\in\{1,2\}$, an action $a_i\in A_i$, and a mixed action $y_{-i}\in Y_{-i}$ with properties (a), (b) and (c). 

Note that it is sufficient to find a player $i\in\{1,2\}$, an action $a_i\in A_i$, and a mixed action $y_{-i}\in Y_{-i}$ with properties (a) and (b), without requiring (c). Indeed, suppose that we have some $a_i$ and $y_{-i}$ with properties (a) and (b). Let $\widehat a_i\in A_i$ be such that (i) $(\widehat  a_i,y_{-i})$ satisfies  properties (a) and (b), and (ii) $r^*_i(\widehat a_i,y_{-i})\geq r^*_i(a_i,y_{-i})$ for every $a_i\in A_i$ such that $(a_i,y_{-i})$ satisfies properties (a) and (b). By the definition of $Y_{-i}$, we also have $r^*_{-i}(\widehat a_i,y_{-i})\geq v_{-i}$. Hence, $\widehat a_i$ and $y_{-i}$ satisfy all properties (a), (b) and (c).

Assume first that condition 1 of Theorem \ref{difficultcase} does not hold. Then, $p(y_1,y_2)>0$ for some $(y_1,y_2)\in Y_1\times Y_2$. By the definition of $Y_1$, we have $r^*_1(y_1,y_2)\geq v_1$. This implies that there is an action $a_1\in A_1$ such that $p(a_1,y_2)>0$ and $r^*_1(a_1,y_2)\geq v_1$. By the definition of $Y_2$, we also have $r^*_2(a_1,y_2)\geq v_2$. 
Hence, $a_1$ and $y_2$ satisfy properties (a) and (b).

Assume now that condition 3 of Theorem \ref{difficultcase} does not hold; the proof is similar if condition 2 does not hold. Then, there are $y_2\in Y_2$ and $a_1\in A_1$ such that $p(a_1,y_2)>0$ and $r^*_1(a_1,y_2)\geq v_1$. By the definition of $Y_2$, we also have $r^*_2(a_1,y_2)\geq v_2$. 
Therefore, $a_1$ and $y_2$ satisfy properties (a) and (b).\medskip

\noindent\textbf{Part 2.} We show that $\Gamma$ admits an $\ep$-equilibrium. 

By Part 1, there is a player $i\in\{1,2\}$, an action $a_i\in A_i$, and a mixed action $y_{-i}\in Y_{-i}$ with properties (a), (b) and (c). Note that $(x^0_i, y_{-i})$ is nonabsorbing due to condition 3(b) of Theorem \ref{TV-thm}.

The game $\Gamma$ admits an $\ep$-equilibrium $\sigma^*$ of the following form:%
\footnote{The structure of the $\ep$-equilibrium is similar to the $\ep$-equilibrium in the proof of Proposition~\ref{lemma eq2}.}
for sufficiently small $\delta>0$, the players play the mixed action profile $\widehat{x}=\left((1-\delta)x^0_i+\delta a_i, y_{-i}\right)$ for the first $N$ stages, where
$N$ is large so that this 
play
leads to absorption with probability close to~1. 
The players switch to $\frac{\ep}{2}$-punishment strategies if one of the following events occurs: (i) no absorption takes place by stage $N$, or (ii) during the play, a statistical test (based on the law of large numbers) detects a deviation by player $-i$: the frequency with which player $-i$ chooses her actions is not sufficiently close to $x^0_{-i}$, including the event that player $-i$ chooses an action that has probability zero under $x^0_{-i}$. 
Provided $\delta$ is sufficiently small, the probability of false detection of deviation is small.

As in the proof of Proposition~\ref{lemma eq2},
$\sigma^*$ is an $\ep$-equilibrium,
provided $\delta$ is sufficiently small and $N$ is sufficiently large.
%
\end{proof}

\section{Open Problems}\label{Sect.Discuss} 

We here present a few open problems that arise from our study.
We proved that every two-player absorbing Blackwell game with tail-measurable payoffs admits an $\ep$-equilibrium, for every $\ep > 0$.
It is natural to ask whether this result can be extended to 
(i) games where monitoring is not perfect
(see, e.g., Shmaya (2011) and Arieli and Levy (2015)),
(ii) games with more than two players, 
(iii) games where the nonabsorbing payoff functions are not necessarily tail-measurable, 
(iv) games where the absorbing payoff depends on the calendarial time or on the history,
and
(v) games with several nonabsorbing states.
We hope that the approach we have taken in this paper will prove useful in establishing these extensions.



\section*{References}

\begin{enumerate}

\item Arieli I. and Levy Y.J. (2015) Determinacy of Games with
Stochastic Eventual Perfect Monitoring, \emph{Games and Economic
Behavior}, \textbf{91}, 166--185.

\item Ashkenazi-Golan G., Flesch J., Predtetchinski A., and Solan
E. (2022) Existence of Equilibria in Repeated Games with Long-run Payoffs. \emph{Proceedings of the National Academy of Sciences} 119, 11: e2105867119.

\item Blackwell D. and Ferguson T.S. (1968) The Big Match, \emph{Annals of Mathematical Statistics}, 33, 159-163.

\item Blackwell D. (1969) Infinite $G_\delta$-Games with Imperfect
Information, \emph{Applicationes Mathematicae}, \textbf{10},
99--101.

\item Blackwell D. and Diaconis P. (1996) A Non-measurable Tail
Set. \emph{Statistics, Probability and Game Theory IMS Lecture
Notes} - Monograph Series Volume 30.

\item Bruy\`{e}re V. (2021) Synthesis of Equilibria in Infinite-Duration: Games on Graphs. In: \emph{ACM SIGLOG News} 8.2, pp. 4–29.

\item Chatterjee K. (2005) Two-Player Nonzero-Sum $\omega$-Regular
Games. In International Conference on Concurrency Theory.
Springer, Berlin, Heidelberg, 413--427.

\item Chatterjee K. (2006) Nash Equilibrium for Upward-Closed
Objectives. In International Workshop on Computer Science Logic).
Springer, Berlin, Heidelberg, 271--286.

\item Chatterjee K. (2007) Concurrent Games with Tail Objectives,
\emph{Theoretical Computer Science}, \textbf{388},
181--198.

\item Chatterjee K. and Henzinger T.A. (2012) A Survey of
Stochastic $\omega$-Regular Games, \emph{Journal of Computer and
System Sciences}, \textbf{78}, 394--413.

\item Everett H. (1957)
Recursive games. In \emph{Contributions to the Theory of Games}, Eds.: Dresher M, Tucker A. W. and Wolfe P., Vol. III, Annals of Mathematics Studies 39, Princeton University Press, Princeton, pp. 47-78.

\item Fink A.M. (1964) Equilibrium in a Stochastic $n$-Person
game. \emph{Journal of Science Hiroshima University}
\textbf{28}(1):89-93.

\item Flesch J., Herings P. J-J., Maes J., Predtetchinski A. (2022) Individual Upper Semicontinuity and Subgame Perfect $\ep$-Equilibria in Games with Almost Perfect Information. \emph{Economic Theory}, \textbf{73}, 695--719.


\item Flesch J. and Solan E. (2022): Equilibrium in Two-player Stochastic Games with Shift-invariant Payoffs. ArXiv:2203.14492.

\item Flesch J., Schoenmakers G. and Vrieze K. (2008)
Stochastic Games on a Product State Space,
\emph{Mathematics of Operations Research} \textbf{33}, 403-420.


\item Flesch J., Thuijsman F. and Vrieze K. (2007)
Stochastic Games with Additive Transitions,
\emph{European Journal of Operational Research} \textbf{179}, 483-497.

\item Gillette D. (1957) Stochastic Games with Zero Stop Probabilities. In \emph{Contributions to the Theory of Games}, Eds.: Dresher M., Tucker A.W., and Wolfe P, Vol. III, Annals of Mathematics Studies 39, Princeton University Press, Princeton, NJ, pp. 179-187.

\item Gr\"adel E. and Ummels M. (2008) Solution Concepts and
Algorithms for Infinite Multiplayer Games, In \emph{New
Perspectives on Games and Interaction}, \textbf{4},
151--178.

\item
Heller Y. (2012)
Sequential Correlated Equilibria in Stopping Games. \emph{Operations Research}, \textbf{60}(1), 209--224.

\item Ja\`{s}kiewicz A. and Nowak A.S. (2016)
Non-zero-sum Stochastic Games, in \emph{Handbook of Dynamic Game Theory}, Eds.: Basar T. and Zaccour G., Springer, Berlin.

\item Le Roux S. and Pauly A. (2014) Infinite Sequential Games
with Real-Valued Payoffs. In Proceedings of the joint meeting of
the twenty-third EACSL annual conference on computer science logic
(CSL) and the twenty-ninth annual ACM/IEEE symposium on logic in
computer science (LICS), 1-10.


\item Maitra A. and Sudderth W. (1993) Borel Stochastic Games with Limsup Payoff. \emph{The Annals of Probability}, 861--885.

\item Maitra A. and Sudderth W. (1998) Finitely Additive
Stochastic Games with Borel Measurable Payoffs,
\emph{International Journal of Game Theory}, \textbf{27},
257--267.

\item Maitra A. and Sudderth W. (2003) Borel Stay-in-a-Set Games,
\emph{International Journal of Game Theory}, \textbf{32},
97--108.


\item Martin D.A. (1998) The Determinacy of Blackwell Games,
\emph{Journal of Symbolic Logic}, \textbf{63}, 1565--1581.




\item Orkin M. (1972) Infinite Games with Imperfect Information,
\emph{Transactions of the American Mathematical Society},
\textbf{171}, 501--507.

\item Rosenthal S. (1975) Nonmeasurable Invariant Sets. \emph{American Mathematical
Monthly} \textbf{82}, 484--491.

\item Shmaya E. (2011) The Determinacy of Infinite Games with
Eventual Perfect Monitoring, \emph{Proceedings of the American
Mathematical Society}, \textbf{139}, 3665--3678.


\item  Shmaya E. and Solan E.  (2004)
Two Player Non Zero-Sum Stopping Games in Discrete Time.
\emph{The Annals of Probability}, \textbf{32}, 2733--2764.

\item Simon R.S. (2007) The Structure of Non-Zero-Sum Stochastic
Games. \emph{Advances in Applied Mathematics}, \textbf{38},
1--26.

\item Simon R.S. (2012) A Topological Approach to Quitting Games,
\emph{Mathematics of Operations Research}, \textbf{37},
180--195.

\item Simon R.S. (2016) The Challenge of Non-Zero-Sum Stochastic
Games, \emph{International Journal of Game Theory}, \textbf{45},
191--204.

\item {Solan E. (1999) Three-Player Absorbing Games.
\emph{Mathematics of Operations Research}, \textbf{24},
669--698.}

\item 
Solan E. (2000)
Stochastic Games with Two Non-Absorbing States.
\emph{Israel Journal of Mathematics}, \textbf{119}, 29--54.

\item
Solan E. (2018)
The Modified Stochastic Game.
\emph{International Journal of Game Theory}, \textbf{47}(4), 1287--1327.

\item Solan E. and Solan O.N. (2020) Quitting Games and Linear
Complementarity Problems, \emph{Mathematics of Operations
Research}, \textbf{45}, 434--454.

\item
Solan E. and Solan O.N. (2021) 
Sunspot Equilibrium in Positive Recursive General Quitting Games.
\emph{International Journal of Game Theory}, \textbf{50}, 891--909.

\item Solan E., Solan O.N., and Solan R. (2020) Jointly Controlled
Lotteries with Biased Coins, \emph{Games and Economic Behavior},
\textbf{119}, 383--391.

\item Solan E. and Vieille N. (2001) Quitting Games.
\emph{Mathematics of Operations Research}, \textbf{26},
265--285

\item  Solan E. and Vohra R. (2001)
Correlated Equilibrium in Quitting Games.
\emph{Mathematics of Operations Research}, \textbf{26}, 601--610.

\item  Solan E. and Vohra R. (2002)
Correlated Equilibrium Payoffs and Public Signalling in Absorbing Games.
\emph{International Journal of Game Theory}, \textbf{31}, 91--121.

\item  Sorin S. (1992) Repeated Games with Complete Information. In Handbook of Game Theory with Economic Applications. Eds. Aumann R., Hart S. (Elsevier), pp. 71–107. 

\item Takahashi M. (1964) Equilibrium Points of Stochastic
Non-cooperative $n$-person Games. \emph{Journal of Science
Hiroshima University} \textbf{28}, 95-99.

\item Thuijsman F. (1992): \emph{Optimality and Equilibria in
Stochastic Games}. CWI-Tract 82, CWI, Amsterdam.

\item Vrieze O.J. and Thuijsman F. (1989) On Equilibria in
Repeated Games with Absorbing States, \emph{International Journal of
Game Theory} \textbf{18}, 293-310.

\item Vervoort M.R. (1996) Blackwell Games, Lecture Notes -
Monograph Series, 369-390.

\item Vieille N. (2000a) Two-player Stochastic Games I: A Reduction,
\emph{Israel Journal of Mathematics}, \textbf{119}, 55--91.

\item Vieille N. (2000b) 
Two-player Stochastic Games II: The Case of Recursive Games,
\emph{Israel Journal of Mathematics}, \textbf{119}, 93--126.
\end{enumerate}
\end{document}